\documentclass[12pt,reqno]{amsart}

\usepackage[utf8]{inputenc}
\usepackage[T1]{fontenc}

\usepackage{lmodern}
\usepackage{microtype}

\usepackage{amsmath, amssymb, amsthm}
\usepackage{geometry}
\usepackage{enumitem}

\usepackage[dvipsnames]{xcolor}

\usepackage[unicode=true]{hyperref}

 \usepackage{cite}
\makeatletter
\newcommand{\citecomment}[2][]{\citen{#2}#1\citevar}
\newcommand{\citeone}[1]{\citecomment{#1}}
\newcommand{\citetwo}[2][]{\citecomment[,~#1]{#2}}
\newcommand{\citevar}{\@ifnextchar\bgroup{;~\citeone}{\@ifnextchar[{;~\citetwo}{]}}}
\newcommand{\citefirst}{\@ifnextchar\bgroup{\citeone}{\@ifnextchar[{\citetwo}{]}}}
\newcommand{\cites}{[\citefirst}
\makeatother

\hypersetup{
  colorlinks,
  citecolor=RoyalBlue,
  linkcolor=RoyalBlue,
  urlcolor=RoyalBlue}
\usepackage{marginnote}
\usepackage{xcolor}
\usepackage{upgreek}
\usepackage{tikz}
\usetikzlibrary{calc}
\usepgflibrary {shadings} 
\usetikzlibrary{intersections}
\usetikzlibrary{patterns}
\usepackage{mathtools}
\usepackage{verbatim}
\usepackage{lipsum}

\numberwithin{equation}{section}

\theoremstyle{plain}
\newtheorem{theorem}{Theorem}
\newtheorem{lemma}{Lemma}[section]
\newtheorem{proposition}{Proposition}[section]

\theoremstyle{definition}
\newtheorem{definition}{Definition}

\newtheorem{remark}{Remark}

\newcommand{\Z}{\mathbb{Z}}
\newcommand{\R}{\mathbb{R}}

\newcommand{\s}{\mathbb{S}}

\newcommand{\x}{\boldsymbol{x}}
\newcommand{\h}{\boldsymbol{h}}
\renewcommand{\d}{\boldsymbol{d}}
\renewcommand{\div}{\operatorname{div}}
\newcommand{\spt}{\operatorname{spt}}
\renewcommand{\j}{\mathfrak{j}}
\newcommand{\J}{\mathfrak{J}}

\DeclarePairedDelimiter{\brk}{(}{)}
\DeclarePairedDelimiter{\abs}{\lvert}{\rvert}
\DeclarePairedDelimiter{\norm}{\lVert}{\rVert}

\DeclarePairedDelimiterX{\inpr}[2]{(}{)}{#1\cdot #2}
\DeclarePairedDelimiterX{\Inpr}[2]{(}{)}{#1:#2}
\DeclarePairedDelimiterX{\intvc}[2]{[}{]}{#1,#2}
\DeclarePairedDelimiterX{\intvl}[2]{(}{]}{#1,#2}
\DeclarePairedDelimiterX{\intvr}[2]{[}{)}{#1,#2}
\DeclarePairedDelimiterX{\intvo}[2]{(}{)}{#1,#2}

\providecommand\st{}
\newcommand\stSymbol[1][]{%
\nonscript\;#1\vert
\allowbreak
\nonscript\;
\mathopen{}}
\DeclarePairedDelimiterX\set[1]\{\}{%
\renewcommand\st{\stSymbol[\delimsize]}
#1
}

\newcommand{\defeq}{\coloneqq}

\newcommand{\Rset}{\mathbb{R}}
\newcommand{\Nset}{\mathbb{N}}
\newcommand{\Zset}{\mathbb{Z}}
\newcommand{\Cset}{\mathbb{C}}
\newcommand{\Sset}{\mathbb{S}}

\newcommand{\dif}{\,\mathrm{d}}

\DeclareMathOperator*{\trace}{tr}

\begin{document}

\title[Stationary \(p\)--harmonic maps]{Stationary \(p\)--harmonic maps approaching planar singular harmonic maps to the circle
% at critical points of the renormalised energy
}

    \author[M. Badran]{Marco Badran}
	\address{ETH Z\"urich, Department of Mathematics, Rämistrasse 101, 8092 Zürich, Switzerland.}
 	\email{marco.badran@math.ethz.ch}
    \author[J. Van Schaftingen]{Jean Van Schaftingen}
\address{Universit\'e catholique de Louvain\\ 
Institut de Recherche en Math\'ematique et Physique\\
Chemin du Cyclotron 2 bte L7.01.01\\
1348 Louvain-la-Neuve\\
Belgium}
\email{Jean.VanSchaftingen@UCLouvain.be}	

\thanks{J. Van Schaftingen was supported by the Projet de Recherche T.0229.21 ``Singular Harmonic Maps and Asymptotics of Ginzburg--Landau Relaxations'' of the Fonds de la Recherche Scientifique--FNRS}
\thanks{M. Badran was supported by the European Research Council under Grant
Agreement No 948029}

% 58E20(1980–now)Harmonic maps, etc. [See also 53C43]
% 49K20(1991–now)Optimality conditions for problems involving partial differential equations
% 35J57(2010–now)Boundary value problems for second-order elliptic systems
% 35J92(2010–now)Quasilinear elliptic equations with p-Laplacian
% 35Q56(2010–now)Ginzburg-Landau equations For optics and electromagnetic theory, see 78A25
% 35B25 - Singular perturbations in context of PDEs
 \maketitle  
    \begin{abstract}
    Given a bounded planar domain $\Omega \subset \mathbb{R}^2$, we show that any singular harmonic map into the circle $\mathbb{S}^1$ corresponding to a topologically nondegenerate critical point of the renormalised energy in the sense of Bethuel, Brezis and H\'elein is a limit of stationary \(p\)-harmonic maps for $p < 2$ as $p \to 2$.
    \end{abstract}

\section{Introduction}

Given a bounded planar domain \(\Omega \subseteq \Rset^2\) with a smooth boundary \(\partial \Omega\) and a boundary datum \(g \colon \partial \Omega \to \Sset^1\), our starting point is the problem of minimizing the Dirichlet energy
\[
 \int_{\Omega} \abs{\nabla u}^2
\]
among suitable maps \(u \colon \Omega \to \Sset^1\).
Classically, such a minimization is performed on the Sobolev space enforcing the constraint on the target and the boundary condition
\[
 W_g^{1,2}\brk{\Omega,\Sset^1}
 \defeq 
  \set[\Big]{u \in W^{1,2}\brk{\Omega,\Rset^2} \st \abs{u} = 1 \text{ almost everywhere in \(\Omega\)} \text{ and } \trace_{\partial \Omega} u = g}.
\]
However, it turns out that if for example \(g\in C^1 \brk{\partial \Omega, \Sset^1}\), then 
\(W_g^{1,2}\brk{\Omega,\Sset^1} \ne \emptyset\) if and only if \(\deg g = 0\) \cite[p. xi]{Bethuel-Brezis-Helein1994}, where \(\deg g\) is Brouwer’s topological degree (or winding number) of \(g\), defined as
\begin{equation}
\label{eq_xi6kui4quahmoh6eeGhohngi}
  \deg g = \frac{1}{2\pi} \int_{\partial \Omega} \det \brk{g, \partial_\tau g}
  = \frac{1}{2\pi} \int_{\partial \Omega} g \times \partial_ \tau g
  = \frac{1}{2 \pi i} \int_{\partial \Omega} \frac{\partial_ \tau g}{g}
\end{equation}
where \(\partial_\tau g\) is the tangential derivative;
the definition \eqref{eq_xi6kui4quahmoh6eeGhohngi} still makes sense under the natural assumption \(g \in W^{1/2, 2}\brk{\partial \Omega, \Sset^1}\) \cites{BoutetdeMonvel_Georgescu_Purice_1991}{Brezis_Nirenberg_1995}[Chapter 12]{Brezis-Mironescu2021}.

Nevertheless, one might expect that a reasonable notion of harmonic map could still be exhibited, despite the infinite energy.
A first approach, adopted in the classical work of Bethuel, Brezis and H\'elein \cite{Bethuel-Brezis-Helein1994} studies the limit as \(\varepsilon \to 0\) \emph{Ginzburg--Landau relaxation} of the Dirichlet energy 
\begin{equation}
\label{eq: GL}
	E_\epsilon(u)=\int_\Omega \abs{\nabla u}^2 +\frac{(1-\abs{u}^2)^2}{2\epsilon^2}
\end{equation}
on the set 
\[
  W_g^{1,2}\brk{\Omega,\Rset^2}
 \defeq 
  \set{u \in W^{1,2}\brk{\Omega,\Rset^2} \st \trace_{\partial \Omega} u = g},
\]
which is not empty provided \(g\) is regular enough.
Although the space \(W_g^{1,2}\brk{\Omega,\Rset^2}\) does not enforce the \(\Sset^1\) constraint anymore, the functional \(E_\epsilon\) heavily penalises values away from the unit circle when \(\epsilon\to 0\).
Bethuel, Brezis and Hélein prove that any family of minimizers \(u_\varepsilon\) have their energies blowing up as $2\pi \abs{\deg(g)} \ln \frac{1}{\varepsilon}$ and admits a subsequence that converges almost everywhere to some  map \(u_*\), with strong convergence away from $n = \abs{\deg \brk{g}}$ distinct singular points \(x_1, \dotsc, x_n\) and with \(u_*\) harmonic outside the singular points, or equivalently, 
\begin{equation}
\label{eq_eeg2tae9Ohsaib2hiegheiSa}
 u_* \brk{x} = \brk[\Big]{\frac{x - x_1}{\abs{x - x_1}}}^{d_1}
 \dotsm \brk[\Big]{\frac{x - x_1}{\abs{x - x_1}}}^{d_n} v_* \brk{x},
\end{equation}
with \(v_* \in C^2 \brk{\Omega, \Sset^1}\) harmonic.
The points \(x_1, \dotsc, x_n\) minimise the \emph{renormalised energy} defined as 
\[
  W \brk{x_1, \dotsc, x_n}
  =
  \lim_{\rho \to 0}
  \int_{\Omega \setminus \bigcup_{j = 1}^n B_\rho \brk{x_i}} \abs{\nabla u_*}^2 - 2 \pi n \ln \frac{1}{\rho}.
\]
The renormalised energy keeps the singularities away from each other and can be computed in terms of a solution to a singular linear elliptic problem \cite[Chapter~I]{Bethuel-Brezis-Helein1994} (see also Appendix \ref{appendix_renormalised}). 
For the case where \(\Omega\) is not simply connected see \cite{Struwe1994} and for general target manifolds see \cite{Monteil_Rodiac_VanSchaftingen_2022, Monteil_Rodiac_VanSchaftingen_2021}.

\medskip

Other relaxations of the Dirichlet problem for minimizing harmonic maps have been proposed and studied. 
The \emph{\(p\)-harmonic relaxation} consists in minimizing the Dirichlet \(p\)-energy
\[
 \int_{\Omega}\abs{\nabla u}^p
\]
on the set 
\[
 W_g^{1,p}\brk{\Omega,\Sset^1}
 \defeq 
  \set{u \in W^{1,p}\brk{\Omega,\Rset^2} \st \abs{u} = 1 \text{ almost everywhere in \(\Omega\)} \text{ and } \trace_{\partial \Omega} u = g}.
\]
When \(p < 2\), we always have \(W_g^{1,p}\brk{\Omega,\Sset^1} \ne \emptyset\) \cite{Hardt_Lin_1987} (see also \cite[Theorem 11.1]{Brezis-Mironescu2021}).
Hardt and Lin \cite{Hardt-Lin1995} showed that the \(p\)-energy of a family of minimisers \(u_p\) behaves as \(2\pi \abs{\deg \brk{g}}/\brk{2 - p}\) as $p\to 2$. Moreover, a subsequence of $\{u_{p}\}$ converges almost everywhere to some  map \(u_*\), with strong convergence away from singular points and \(u_*\) harmonic outside the singular points minimising the same renormalised energy. 
Similar results hold for a general target manifold \cite{VanSchaftingen_VanVaerenbergh_2023}.
Relaxation schemes involving more classes of integrands going beyond the powers have also been analysed \cite{Irving_VanVaerenbergh_2024}.
Higher dimensional and codimensional situations have also been studied \cite{Badran2024,Hardt-Lin-Wang1997,Stern2021}.

\medskip

A further question is to determine whether non-minimising singular harmonic maps can be obtained through a similar procedure.
For the Ginzburg--Landau relaxation, several works have shown that 
any \(u_*\) of the form \eqref{eq_eeg2tae9Ohsaib2hiegheiSa} is the limit of critical points of the Ginzburg-Landau energy \eqref{eq: GL} provided that \(\brk{x_1, \dotsc, x_n}\) is in some sense a nondegenerate critical point of the renormalised energy and \(\abs{d_j} = 1\)
\cite{Lin_1995,Lin_Lin_1997,delPino_Kowalczyk_Musso_2006,delPino_Felmer_1997,Pacard_Riviere_2000,Stern2021}.

\medskip

We are interested in the approximation of singular harmonic maps at nondegenerate critical points of the renormalised energy by \(p\)-harmonic map. 
Recall that $u\in W^{1,p} \brk{\Omega, \Sset^1}$ is a weak $p$-harmonic map is a critical point of $E_p$ with respect to outer variations: for every \(v\in C^\infty_c\brk{\Omega,\Rset^2}\)
\begin{equation*}
	\frac{d}{dt}\Big\vert_{t=0}\int_\Omega\abs[\big]{\nabla\brk[\big]{\tfrac{u+tv}{|u+tv|}}}^p = 
	p \int_{\Omega} \abs{\nabla u}^{p - 2} \Inpr{\nabla u}{\nabla v}
	- \abs{\nabla u}^p \inpr{u}{v}
	=0,
\end{equation*} 
that is, 
\begin{equation}
\label{eq_eiy5ootith2fa4oo0aeThiej}
 -\operatorname{div} \brk[\Big]{\frac{\nabla u}{\abs{\nabla u}^{2 - p}}}
 = \abs{\nabla u}^p u,
\end{equation}
where  $\inpr{-}{-}$ and \(\Inpr{-}{-}\) denote the inner products on \(\Rset^2\) and on \(\Rset^{2 \times 2}\) respectively.

The condition of being a weak $p$-harmonic map is quite loose  (see \cite{Hardt-Chen1995}, or Proposition \ref{prop: weak pharm} in the case of the circle): if \(\Omega\) is connected and 
if \(\partial \Omega = \bigcup_{\ell = 0}^m \Gamma_\ell\), with \(\Gamma_\ell\) connected and \(\Gamma_0\) the outer boundary, and if for some $n\geq 1$, $\x=\{x_1,\dots,x_n\}$ is an element of the $n$-configuration space of $\Omega$
\begin{equation*}
	\Omega^n_\ast\coloneqq \left\lbrace(x_1,\dots,x_n)
	\in \Omega^n:x_j\ne x_k\text{ for }j\ne k\right\rbrace.
\end{equation*}
and the integers $\d =(d_1,\dots,d_n)\in\Z^n$ satisfy the compatibility condition
\begin{equation}
\label{eq: compatibility}
  \deg (g\vert_{\Gamma_0}) = \sum_{j = 1}^n d_j + \sum_{\ell = 1}^m \deg (g\vert_{\Gamma_\ell}),
\end{equation}
then 
there exists a \(p\)-harmonic map \(u_{\x, p} \in C \brk{\Omega \setminus \set{x_1, \dotsc, x_n}, \Sset^1}\) and \(\deg u\vert_{\partial B_\rho\brk{x_j}} = d_j\).
(This wild abundance of \(p\)-harmonic mappings has some similarities with everywhere singular weak harmonic maps in higher dimension \cite{Riviere_1995} and weakened notions of harmonic maps \cite{Almeida_1995,Ji_2001}.)

A stricter notion of criticality for $p$-harmonic maps is the stationarity condition, namely criticality with respect to inner variations
\begin{equation}\label{eq: stationarity}
	\frac{d}{dt}\Big\vert_{t=0}\int_\Omega\left|\nabla\left(u\circ\varphi_t\right)\right|^p=0,
\end{equation}
for every smooth family of diffeomorphisms $\varphi_t \colon \Omega \to \Omega$ such that \(\varphi_0 = \operatorname{id}\) and \(\varphi_t = \operatorname{id}\) in a neighbourhood of \(\partial \Omega\); it turns out that many of the \(p\)-harmonic maps with prescribed singularities are not stationary.

Given a point \(\brk{\x_*, \d}\) --- namely a collection of points in the configurations space with associated degrees --- we consider a map \(u_*\) of the form 
\eqref{eq_eeg2tae9Ohsaib2hiegheiSa} with \(v_*\) harmonic, so that \(u_*\) is harmonic outside \(\Omega \setminus \set{x_{*, 1}, \dotsc, x_{*, n}}\) and \(u_* \vert_{\partial \Omega} = g\).
When \(\Omega\) is simply connected, the map \(u_*\) is prescribed completely; otherwise, any map of the form \(u_* e^{i \varphi}\) with \(\varphi \colon \Omega \to \Rset\) harmonic and \(\varphi\vert_{\partial \Omega} \in 2 \pi \Zset\) will have the same properties and we \emph{fix} thus such a mapping. 
If \(\x\) is close enough to \(\x_*\) we can find a continuous family of mappings \(u_{\x}\) having these properties and such that \(u_{\x_*} = u_{\x}\) and we can define the \emph{renormalized energy}
\begin{equation}
\label{eq_ahSh8shaiphaiha3oocheich}
  W \brk{\x}
  \defeq
  \lim_{\rho \to 0}
  \int_{\Omega \setminus \bigcup_{j = 1}^d B_\rho \brk{x_i}} \abs{\nabla u_{\x}}^2 - 2 \pi \sum_{i = 1}^n d_j^2  \ln\frac{1}{\rho}
  \\;
\end{equation}
one can prove that \(W\) is of class \(C^1\) around \(\x\).

The following result, proved in \S\ref{sec: Stat}, shows how the renormalised energy is related to stationary $p$-harmonic maps.

\begin{proposition}
\label{proposition_critical}
If $\brk{u_k}_{k \in \Nset}$ is a sequence of stationary $p_k$-harmonic maps to $\s^1$, for $p_k\to 2$,  if \(u_k \to u_*\) almost everywhere and if 
\(
  \nabla u_k/\abs{\nabla u_k}^{p - 2} \to \nabla u_*
\) in \(L^2_{\mathrm{loc}} \brk{\Omega \setminus \set{x_{*, 1}, \dotsc, x_{*, n}}}\) and if \(u_*\) is a singular harmonic map, then $\nabla W \brk{\x_*} = 0$.
\end{proposition}

In the conclusion of Proposition~\ref{proposition_critical}, the singular harmonic map \(u_*\) can only have singularities at the points \(x_{*, 1}, \dotsc, x_{*, n}\), although some of them could have degree \(0\) and be removable.

\medskip

We say that \(\x\) is a \emph{topologically nondegenerate critical point} of \(W\) whenever \(\nabla W \brk{\x}= 0\) and for a sufficiently small ball \(B_{\delta} \brk{\x} \subseteq \Omega^n_*\), \(\nabla W \ne 0\) on \(\partial B_{\delta} \brk{\x}\) and \(\deg \brk{\nabla W, B_\delta \brk{\x}, 0} \ne 0\), where \(\deg\) denotes the \(2n\)-dimensional Brouwer degree.

Our main result states that from any nondegenerate critical point $(\x_*,\d)$ of the renormalised energy one can produce a collection of stationary $p$-harmonic maps, whose singular set is close to $\x_*$.

\begin{theorem}\label{thm: main}
	Let $\Omega$ be a planar domain, $g\colon\partial\Omega\to\s^1$, let \(u_*\) be a harmonic map of the form 
\eqref{eq_eeg2tae9Ohsaib2hiegheiSa} such that \(u_* \vert{\partial \Omega} = g\).
If \(\x_* = \set{x_{1, *}, \dotsc, x_{n, *}}\in\Omega^n_*\) is a topologically nondegenerate critical point of the renormalised energy \(W\) given by \eqref{eq_ahSh8shaiphaiha3oocheich}, then, there exists $p_\circ\in(1,2)$ such that for every $p\in \intvo{p_\circ}{2}$ there exists a configuration \(\x_p = \brk{x_{1, p}, \dotsc, x_{n, p}}\), with $\x_{p}\to \x_*$ as $p\to2$ for every $j\in\{1,\dots,n\}$, and a stationary $p$-harmonic map $u_p\in\Omega\to\s^1$ such that \(u_p \in C \brk{\Omega \setminus \set{x_{1, p}, \dotsc, x_{n, p}}, \Sset^1}\) and 
\(u_p \to u_*\) locally uniformly in \(\Omega \setminus \set{x_{1, *}, \dotsc, x_{n, *}}\). 
\end{theorem}

In particular, for any $j\in \{1,\dots,n\}$ and $\delta>0$ sufficiently small, we have 
\begin{equation*}
  \deg \brk{u_p|_{\partial B_{\delta}\brk{x_{j,p}}}}=d_j.
\end{equation*}

Prescribing singularities for weakly $p$-harmonic maps to spheres is known to be possible. Our result shows that the stationarity condition is generally less flexible and strictly related to the position of the singularities. 

\medskip

The proof consists of two steps. First, we show that to any admissible singular set $(\x,\d)$ we can associate a weak $p$-harmonic map given as a phase perturbation of a given fixed map, encoding the singularities. Secondly, we show that the stationarity condition correspond to a finite dimensional problem, and we relate such problem to the criticality of the singular set for the renormalised energy in the limit $p\to 2$. Lastly, a topological degree argument shows that this finite dimensional problem can be solved for values of $p$ sufficiently close to 2.

\section{Existence of \texorpdfstring{$p$}{p}-harmonic maps with prescribed singularities}

We are going to construct \(p\)-harmonic maps next to singular harmonic maps, that we formally define as follows.

\begin{definition}
	Given a domain $\Omega\subset\R^2$, we say that $u_*$ is a \emph{singular harmonic map} with singular set $(\x,\d)$ whenever
there exists a harmonic map \(v_* \in C^2 \brk{\Omega, \Sset^1}\) such that 
	\begin{equation}
	\label{eq_NeiSh7nooch7chao6maikai6}
 u_* \brk{x} = \brk[\Big]{\frac{x - x_1}{\abs{x - x_1}}}^{d_1}
 \dotsm \brk[\Big]{\frac{x - x_1}{\abs{x - x_1}}}^{d_m} v_* \brk{x},
\end{equation}	
\end{definition}

If $\Omega$ is simply connected, the canonical harmonic map constructed in \cite[\S I.3]{Bethuel-Brezis-Helein1994} is a singular harmonic map. For a general domain $\Omega$, maps with the same  properties still exist, but lack uniqueness.

We define the vector field  
\begin{equation}
\begin{split}
	\j u &\coloneqq (\nabla u,iu)_{\Cset}= u^1 \nabla u^2-
	u^2 \nabla u^1\\
	&= \det \brk{u, \nabla u}
	= u \times \nabla u
	= - i \Bar{u} \nabla u = - i u^{-1} \nabla u,
\end{split}
\end{equation}
where the vector operations are performed on the target side in \(\Cset \simeq \Rset^2\).
The last identities follow from the fact that $(\nabla u,u)_{\Cset}=\frac12\nabla \abs{u}^2=0$, so that 
\begin{align*}
	\nabla u&= iu \brk{\j u}&
	&\text{ and }&
	\abs{\nabla u} & = \abs{\j u}
\end{align*}
It also satisfies 
\begin{align*}
 \j e^{i \varphi} &= \nabla \varphi&
 &\text{ and}
 &\j \brk{uv} & =\j u + \j v ,
\end{align*}
while 
\begin{equation*}
 \j  \brk[\Big]{\frac{x - x_0}{\abs{x - x_0}}} = \frac{\brk{x - x_0}^\perp}{\abs{x - x_0}^2}
\end{equation*}
with \(\brk{h_1, h_2}^\perp = \brk{-h_2, h_1}\), which is divergence-free.
It follows thus from \eqref{eq_NeiSh7nooch7chao6maikai6} that 
\begin{equation}
\label{eq: expansion vx}
 \brk{\j u_*} \brk{x} = 
 \brk{\j v_*} \brk{x} + \sum_{j = 1}^n d_j \frac{\brk{x - x_j}^\perp}{\abs{x - x_j}^2}
\end{equation}
The harmonicity of \(v_*\) implies that 
\begin{equation*}
 -\operatorname{div} \brk{\j u_*} = 0.
\end{equation*}
Defining the quantity
\begin{equation}
 \J u \defeq \operatorname{curl} \brk{\j u},
\end{equation}
one has thus 
\begin{equation}
\label{eq_eisho5ugedueyievahChaich}
 \J u_* = \sum_{j = 1}^n 2\pi d_j \delta_{x_j},
\end{equation}
see \cite{Alberti-Baldo-Orlandi2003,Brezis-Mironescu2021}.

The following result is well known but we sketch its proof for completeness.

\begin{proposition}\label{thm: singular HM}
	To any configuration $(\x_*,\d)$ and boundary datum $g$ satisfying \eqref{eq: compatibility} there exists a singular harmonic map \(u_{*}\) with the prescribed singular set and boundary datum. \\
	Moreover, the set of singular harmonic maps satisfying the conditions is given by 
	\[
	 \set{u_{*} e^{i \varphi}\st \varphi \in C^\infty \brk{\Omega, \Rset}, \Delta \varphi = 0 \text{ and } \varphi \brk{\partial \Omega} \subseteq 2 \pi \Zset}.
	\]
\end{proposition}

When \(\Omega\) is simply-connected, one has \(\varphi = 2 \pi k\) for some \(k \in \Zset\) and thus \(u_{*}\) is unique and corresponds to the canonical harmonic map \cite{Bethuel-Brezis-Helein1994}. 

	\begin{proof}[Proof of Proposition~\ref{thm: singular HM}]
		Consider the map 
		\begin{equation*}
			w_{*}\coloneqq\prod_{j=1}^n\left(\frac{x-x_j}{|x-x_j|}\right)^{d_j} \prod_{\ell=1}^n\left(\frac{x-a_\ell}{|x-a_\ell|}\right)^{e_\ell},
		\end{equation*}
		where \(a_\ell \in \Rset^2 \setminus \Omega\) are points chosen inside \(\Gamma_\ell\) and \(e_\ell = \deg g\vert_{\Gamma_\ell}\) for \(\ell \in \set{1, \dotsc, m}\).
		By \eqref{eq: compatibility}, for every \(\ell \in \set{0, \dotsc, m}\), the map $g w_{*}^{-1}\colon \Gamma_{\ell} \to\s^1$ is continuous and has degree \(0\) on $\Gamma_\ell$ and thus $g w_{*}^{-1}\vert_{\partial\Omega}=e^{ih}$ for some \(h \in C \brk{\partial \Omega, \Rset}\) (see \cite[Lemma 1.3]{Brezis-Mironescu2021}). 
		Next, we let $\varphi$ be the (smooth) harmonic extension of $h$ in $\Omega$ and  we set 
		\(
			u_{*}\defeq e^{i\varphi}\, w_{*}
		\).
		
		For the second part, given singular harmonic maps \(u_{*}\) and \(u_{*}'\)
		and letting \(v_*\) and \(v_*'\) be the associated harmonic maps, 
	    we note that \(w \defeq u_*^{-1}u_*' = v_*^{-1} v_*'\) is harmonic and equal to \(1\) on \(\partial \Omega\). 
	    In particular, \(w\) has degree zero on every connected component of the boundary  and hence can be written as \(w = e^{i \varphi}\), where $\varphi$ is harmonic and locally constant on the boundary, with values in $2\pi\Z$.	
		\end{proof}
		
\begin{proposition}\label{thm: singular HM continous}
Given a singular harmonic map \(u_*\) with singularities $(\x_*,\d)$,
there exists for \(\x\) close to \(\x_*\) singular harmonic maps \(u_{\x}\) with singularties $(\x,\d)$ and the same boundary data depending continuously on \(\x\) in \(W^{1, p}\brk{\Omega, \Sset^1}\).
\end{proposition}
\begin{proof}
Letting \(v_*\) be given by \eqref{eq_NeiSh7nooch7chao6maikai6}, we define \(h \colon \partial \Omega \to \Sset^1\) by 
\[
  h_{\x}\brk{x} \defeq \prod_{i= 1}^n \brk[\Big]{\frac{x - x_{j,*}}{\abs{x - x_{j, *}}}\frac{\abs{x - x_j}}{x - x_j}}^{d_i}.
\]
Since \(h_{\x} \to h_{\x_*} = \operatorname{id}\) as \(\x \to \x_*\), if \(\x\) is close enough to \(\x_*\), we can find \(\theta_{\x} \in C \brk{\partial \Omega, \Rset}\) depending continuously on \(\x\) such that \(h_{\x} = e^{i \theta_{\x}}\).
Letting \(\varphi_{\x}\) be the harmonic extension of \(\theta_{\x}\) and setting \(v_{\x} \defeq e^{i \varphi_{\x}}\), we observe that \(u_{\x}\) defined by 
\[
 u_{\x} \brk{x} \defeq \brk[\Big]{\frac{x - x_1}{\abs{x - x_1}}}^{d_1}
 \dotsm \brk[\Big]{\frac{x - x_1}{\abs{x - x_1}}}^{d_m} v_{\x} \brk{x}
\]
satisfies the conclusion.
\end{proof}

We prove the following result.

\begin{proposition}\label{prop: weak pharm}
	For every $p\in(1,2)$, every smooth boundary datum $g$ and every continuous family of singular harmonic maps \(u_{\x}\),
	there is family of weakly \(p\)-harmonic maps 
		\begin{equation*}
			u_{p,\x}\in\mathcal{C}_{p,\x}\coloneqq \set[\big]{u=e^{i\phi}u_{\x}:\phi\in W^{1,p}_0\brk{\Omega,\s^1}} \subseteq W_g^{1,p}\brk{\Omega,\s^1};
		\end{equation*}
	depending continuously on $\x$ in \(W^{1,p}\brk{\Omega,\s^1}\).
\end{proposition}

In particular, 
\begin{equation}\label{eq: degree upx around sing}
	\deg\brk{u_{p,\x}\vert_{\partial B_{r}(x_j)}}=d_j
\end{equation}
where the degree in \eqref{eq: degree upx around sing} is well defined for almost every small radius $r>0$ for any $u\in W^{1,p}(\Omega,\s^1)$ as a consequence of Fubini's theorem. Indeed, letting $\delta>0$ be small
	\begin{equation*}
		\int_0^{\delta}rdr\int_{0}^{2\pi}|\nabla^\top u(r,\theta)|^pd\theta \leq \int_{B_\delta(x_j)}|\nabla u|^p<\infty
	\end{equation*}
	where $(r,\theta)$ are polar coordinates centred in $x_j\in\Omega$ and $\nabla^\top$ is the tangential gradient. This means that for a.e. $r\in (0,\delta)$ the map $u(r,\cdot)$ belongs to $W^{1,p}(\s^1,\s^1)$, thus it is continuous and the degree is well defined.

Equation \eqref{eq_eiy5ootith2fa4oo0aeThiej} is equivalent to having 
\[
  -\operatorname{div} \brk[\Big]{\frac{\j u}{\abs{\j u}^{2 - p}}} = 0
\]
whereas the information on the degrees yields \eqref{eq_eisho5ugedueyievahChaich}.

	\begin{proof}[Proof of Proposition \ref{prop: weak pharm}]
		Let $u_{\x}$ be the singular harmonic map arising from Theorem \ref{thm: singular HM}. We look for $u_{p,\x}$ as a minimiser of the $W^{1,p}$ energy in the set
		\begin{equation*}
			\mathcal{C}_{p,\x}\coloneqq \set[\big]{u=e^{i\phi}u_{\x}:\phi\in W^{1,p}_0\brk{\Omega,\s^1}}.
		\end{equation*}
		Remark that 
		\begin{equation}\label{eq: minimisation}
			\inf_{u\in \mathcal{C}_{p,\x}}\int_\Omega|\nabla u|^p=\inf_{\phi\in W^{1,p}_0(\Omega,\s^1)}\int_\Omega|\nabla \phi+\j u_{\x}|^p
		\end{equation}
		and that the integrand on the right-hand side is convex with respect to \(\phi\) --- so we can find a minimiser $\phi_{p,\x}$ and denote $u_{p,\x}=e^{i\phi_{p,\x}}u_{\x}$ the corresponding minimiser in $\mathcal{C}_{p,\x}$. The boundary condition follows automatically since $\phi_{p,\x}$ vanishes at the boundary. To show the weak $p$-harmonicity, let $w\in C^1_c(\Omega,\R^2)$ and, for $|t|$ small enough, consider the outer variation
		\begin{equation*}
			u_t\coloneqq\frac{u_{p,\x}+tw}{|u_{p,\x}+tw|}.
		\end{equation*}
		One checks directly that $u_t=u_{p,\x}+t\hat{w}( iu_{p,\x})+O(t^2)$, where $\hat{w}=(w,u_{p,\x})$. Consequently, we have $\j u_t=\j u_{p,\x}+t\nabla \hat w+O(t^2)$ and 
		\begin{align*}
			\frac{d}{dt}\Big\vert_{t=0}\int_{\Omega}|\nabla u_t|^p&=\frac{d}{dt}\Big\vert_{t=0}\int_{\Omega}|\j u_t|^p\\
			&=p\int_\Omega|\j u_{p,\x}|^{p-2}\j u_{p,\x}\cdot\nabla\hat w\\
			&=p\int_\Omega|\nabla \phi_{p,\x}+\j u_{\x}|^{p-2}(\nabla \phi_{p,x}+\j u_{\x})\cdot\nabla \hat w
		\end{align*}
		and the last quantity vanishes since $\phi_{p,\x}$ minimises \eqref{eq: minimisation}. Finally, we prove continuity in $\x$. Consider a sequence of configurations $\x_k\to \x_0$ and note that by construction $u_{\x_k}\to u_{\x_0}$ in $W^{1,p}$ (which implies $\j u_{\x_k}\to \j u_{\x_0}$ in $L^{p}$), so we only need to show that $\phi_{p,\x_k}\to \phi_{p,\x_0}$ in $W^{1,p}$. By Poincar\'e's inequality, it is enough to show that $\nabla \phi_{p,\x_k}\to \nabla \phi_{p,\x_0}$ in $L^p$. Let 
		\begin{equation*}
			F_{\x}(\phi)\coloneqq \int_{\Omega}|\nabla \phi+ju_{\x}|^p.
		\end{equation*}
		We claim that $F_{\x_k}(\phi_{p,\x_k})\to F_{\x_0}(\phi_{p,\x_0})$. Indeed, by lower semicontinuity we have 
		\[
		  F_{\x_0}(\phi_{p,\x_0})\leq\liminf_{k\to \infty}F_{\x_k}(\phi_{p,\x_k})
		\]
		and by taking the limsup of the comparison inequality 
		\begin{equation*}
			F_{\x_k}(\phi_{p,\x_k})\leq F_{\x_k}(\phi_{p,\x_0})
		\end{equation*}
		we find 
		\[
			\limsup_{k \to \infty} F_{\x_k}(\phi_{p,\x_k})\leq F_{\x_0}(\phi_{p,\x_0}). 
		\]
		        By uniform convexity of the $L^p$ space, the combination of the  weak convergence $\nabla \phi_{p,\x_k}+\j u_{\x_k}\rightharpoonup \nabla \phi_{p,\x_0}+\j u_{\x_0}$ with the convergence of the norms implies strong convergence. 
		        We have then 
		\begin{equation*}
			\|\nabla \phi_{p,\x_k}-\nabla \phi_{p,\x_0}\|_{L^p}
			\leq C\left(\|\nabla (\phi_{p,\x_k}+\j u_{\x_k})
			-\nabla (\phi_{p,\x_0}+\j u_{\x_0})\|
			_{L^p}+\|\j u_{\x_k}-\j u_{\x_0}\|_{L^p}.\right);
		\end{equation*}
		letting $k\to\infty$ yields the result.
	\end{proof}
	
\section{Stationarity near critical points of \texorpdfstring{$W_g$}{Wg}}\label{sec: Stat}
The goal of this section is to show that some of the weak $p$-harmonic maps arising from Proposition \ref{prop: weak pharm} satisfy the stationarity condition \eqref{eq: stationarity}.
Recall that the stress-energy tensor of a map $u\in W^{1,p}(\Omega,\s^1)$ is defined as 
\begin{equation*}
	S_p(u)\coloneqq p|\nabla u|^{p-2}\nabla u\otimes \nabla u-|\nabla u|^pI
\end{equation*}
where $I$ is the identity matrix and $(\nabla u\otimes \nabla u)_{ij}=(\partial_iu,\partial_ju)$.
If \(\varphi_t\) is a smooth family of diffeomorphisms such that \(\varphi_0 = \operatorname{id}\) and \(\varphi_t = \operatorname{id}\) in a neighbourhood of \(\partial \Omega\), the first inner variation of the energy is related to the stress energy tensor by
\begin{equation}\label{eq: inner variation eq}
	\frac{d}{dt}\Big\vert_{t=0}\int_\Omega|\nabla u\circ \varphi_t|^p=\int_\Omega \Inpr{S_p(u)}{D\partial_t \varphi_0}.
\end{equation}
%where $\langle A,B\rangle\coloneqq \Tr(A^TB)$ is the Frobenius product of matrices, 
Note that if $u\in W^{1,p}(\Omega,\s^1)$, then $S_p(u)\in L^{1}(\Omega,\R^{2\times 2})$. In particular, we can define the vector-valued distribution
\begin{equation*}
	\div S_p(u)\in \mathcal{D}'(\Omega)
\end{equation*} 
acting on vector fields $X\in C^\infty_c(\Omega,\R^2)$ by 
\begin{equation*}
	\div S_p(u)(X)\coloneqq\int_\Omega \left\langle S_p(u), DX\right\rangle.
\end{equation*}
\begin{lemma}
	There are $n$ pairs of coefficients $c_j(p,\x)=(c^1_j(p,\x),c^2_j(p,\x))$, \(j \in \set{1, \dotsc, n}\) such that 
	\begin{equation*}
		\div S_p(u_{p,\x})=\sum_{j=1}^nc_j(p,\x)\delta_{x_j}
	\end{equation*}
	in the sense that 
	\begin{equation*}
	\int_\Omega \left\langle S_p(u), DX\right\rangle
	= \sum_{j=1}^n \inpr{c_j\brk{p,\x}}{X \brk{x_j}}.
	\end{equation*}
	Moreover, the maps $\x \mapsto c_j\brk{p,\cdot}$ are continuous, for $j\in \set{1,\dotsc,n}$.
	\begin{proof}
		Let $X$ be supported away from $\x$ and let $\Phi^X_t$ is the associated flow, solving 
\begin{equation*}
\begin{cases}
	\partial_t\Phi^X_t=X\circ\Phi^X_t,\\
	\Phi_0^X=\operatorname{id}.
\end{cases}
\end{equation*}
		Then we claim that $u_{p,\x}\circ \Phi^X_t\in\mathcal{C}_{p,\x}$. Indeed, first note that 
		\begin{equation*}
			\J((u_{\x}\circ \Phi^X_t)u_{\x}^{-1})=\J(u_{\x}\circ \Phi^X_t)-\J u_{\x}=2\pi\sum_{i=1}^nd_j\delta_{x_j}-2\pi\sum_{i=1}^nd_j\delta_{x_j}=0.
		\end{equation*}
		Then, we observe that the $W^{1,2}$ energy of $(u_{\x}\circ \Phi^X_t)u_{\x}^{-1}$ can be made arbitrarily small up to reducing $|t|$. Indeed, by direct differentiation and the mean value theorem
		\begin{equation*}
			\|\nabla (u_{\x}\circ \Phi^X_t)u_{\x}^{-1}\|_{L^2(\Omega)}\leq C|t|.
		\end{equation*}
		where $C$ depends only on $X$ and $\Omega$. By \cite[Proposition 14.3]{Brezis-Mironescu2021}, $(u_{\x}\circ \Phi_X^t)u_{\x}^{-1} = e^{i\psi_t}$ with \(\psi_t \in W^{1,2}\brk{\Omega, \Rset}\). In particular 
		\begin{equation*}
			u_{p,\x}\circ \Phi^X_t=e^{i\phi_{p,\x}\circ\Phi^X_t}(u_{\x}\circ \Phi^X_t)=e^{i(\phi_{p,\x}\circ\Phi^X_t+\psi_t)}u_{\x}.
		\end{equation*}
		Thus, since $u_{p,\x}$ is minimiser in the class $\mathcal{C}_{p,\x}$,
		\begin{equation*}
			\frac{d}{dt}\Big\vert_{t=0}\int_\Omega|\nabla u_{p,\x}\circ \Phi^X_t|^p=0
		\end{equation*}
		and by \eqref{eq: inner variation eq} we get 
		\begin{equation*}
			\div S_p(u_{p,\x})(X)=0.
		\end{equation*}
		By a result in distribution theory (see for example \cite[Theorem 6.25]{Rudin_1991}), a distribution supported in a finite union of points is a finite sum of derivatives of Dirac deltas over the points, and since $\div S_p(u_{p,\x})$ is the distributional divergence of an $L^1$ matrix we infer that the derivatives are all of order zero. More precisely
		\begin{equation*}
			\div S_p(u_{p,\x})=\sum_{j=1}^nc_j\delta_{x_j}
		\end{equation*}%
		%\todo[inline,color=GreenYellow]{JVS (2025-04-25) We have here zero order derivative, wheras we have written above first-order. There should be a way to fill the gap.}%
		for some $c_j=c_j(p,\x)$. Lastly, continuity of $c_j(p,\cdot)$ follows by continuity of $u_{p,\cdot}$ in $W^{1,p}\brk{\Omega, \Sset^1}$ (see Proposition~\ref{prop: weak pharm}).
	\end{proof}
\end{lemma}
In what follows we denote 
\begin{equation*}
	B_\delta(\x)\coloneqq \prod_{j=1}^nB_\delta(x_j).
\end{equation*}

The following result states that the coefficients $c_j$ approach asymptotically the derivatives of the renormalised energy, locally uniformly in the configuration space. This is the main tool to prove Theorem \ref{thm: main} and we postpone its proof to the end of the section. 
\begin{proposition}
\label{prop: unif conv}
	Let $\x_*\in \Omega^n_\ast$ and 
	\begin{equation*}
		\delta<\tfrac12\min\left(\min_{j\ne k}|x_{*j}-x_{*k}|,\min_{j}\operatorname{dist}(x_{*j},\partial\Omega)\right).
	\end{equation*}
	Then, $c(p,\cdot)\to\nabla W$, uniformly in $B_\delta(\x_*)$, as $p\nearrow 2$. 
\end{proposition}
Assuming for the moment the validity of Proposition \ref{prop: unif conv}, we can conclude the proof of Theorem \ref{thm: main}.

\begin{proof}[Proof of Theorem \ref{thm: main}]
Let $\x_*$ be a topologically nondegenerate critical point of $W$. Then, we can find an open neighbourhood $\mathcal{U}= B_\delta(\x_*)\subset \Omega^n_\ast$ such that 
\begin{equation*}
	\partial \mathcal{U} \cap \{\nabla W=0\}=\emptyset
\end{equation*}
and 
\begin{equation*}
	\deg(\nabla W,\mathcal{U},0)\ne 0.
\end{equation*}
By the properties of topological degree (see for example \cite[\S3.1]{Ambrosetti-Malchiodi2007}), $\deg(\cdot,\mathcal{U},0)$ is continuous with respect to uniform convergence. In particular, for $p$ sufficiently close to $2$, 
\begin{equation*}
	\deg(c(p,\cdot),\mathcal{U},0)\ne 0.
\end{equation*}
Then, it follows that there is some $p_\circ\in (1,2)$ such that for every $p\in (p_\circ,2)$ there exists $\x=\x(p)$ such that 
\begin{equation*}
	c(p,\x(p))=0.
\end{equation*}
By the discussion above this means that $u_{p,\x(p)}$ is a stationary $p$-harmonic map.  
\end{proof}

The first step in the proof of Proposition \ref{prop: unif conv} is showing that the gradients $\nabla\phi_{p,\x}$ of the corrections arising from Proposition \ref{prop: weak pharm} are small in a suitable sense, when $p\nearrow 2$. In all the results below $\x_*$ and $\delta>0$ will be as in Proposition \ref{prop: unif conv}.
	\begin{lemma}\label{lem: aux}
	Let $\delta>0$ be sufficiently small. Then, it holds
		\begin{equation*}
			\int_{\Omega}|\nabla\phi_{p,\x}|^p \leq C \brk{2 - p}.
		\end{equation*}
			as $p\nearrow 2$, where $C>0$ is uniform in $\x\in B_\delta(\x_*)$.
	\end{lemma}

			\begin{proof}
			We start by recording the trivial bound 
			\begin{equation}\label{eq: trivial bound}
				\int_{\Omega}|\nabla\phi_{p,\x}|^p\leq C\int_\Omega|\j u_{\x}|^p\leq \frac{C}{2-p}
			\end{equation}  
			with a constant independent on $\x\in B_\delta(\x_*)$ and $p$, as a simple consequence of testing the minimality in \eqref{eq: minimisation} with $\phi=0$, the triangle inequality and the form \eqref{eq_NeiSh7nooch7chao6maikai6} of singular harmonic maps.
			Secondly, observe that 
			\begin{equation}
			\label{eq_aisho2eihies4aethee5zeeJ}
				\begin{split}
					\int_{\Omega}|\nabla\phi_{p,\x}|^p&=2 \int_0^1(1-t)\int_{\Omega}|\nabla\phi_{p,\x}|^p \dif t \\
					&\leq \left(\int_0^1\int_\Omega\frac{1-t}{|\j u_{\x}+ t\nabla\phi_{p,\x}|^{2-p}}|\nabla\phi_{p,\x}|^2\dif t \right)^{\frac{p}2}\\
					&\qquad \times \left(\int_0^1\int_\Omega(1-t)|\j u_{\x}+t\nabla\phi_{p,\x}|^p\dif t \right)^{\frac{2-p}{2}},
				\end{split}
			\end{equation}
			with the understanding that the integrand in the first integral on the right-hand side is \(0\) where \(\j u_{\x} + t \nabla\phi_{p,\x} = 0\).
			Note that by \eqref{eq: trivial bound} we get a uniform bound 
			\begin{equation}
			\label{eq_vuphaejuxa8quae6Uv9theuB}
			\begin{split}
				\left(\int_0^1\int_\Omega(1-t)|\j u_{\x}+ t\nabla\phi_{p,\x}|^p\right)^{\frac{2-p}{2}}&
				\leq \left(2^{p - 1} \int_\Omega|\j u_{\x}|^p+|\nabla\phi_{p,\x}|^p\right)^{\frac{2-p}{2}}\\
				&\leq C\brk[\Big]{\frac{1}{2-p}}^{\frac{2-p}{2}}\\
				&\leq C.
			\end{split}
			\end{equation}
			Thus, we find by \eqref{eq_aisho2eihies4aethee5zeeJ} and \eqref{eq_vuphaejuxa8quae6Uv9theuB}
			\begin{equation}
			\label{eq_phaib9eewieThee5mirahqu1}
\brk[\Big]{\int_{\Omega}\abs{\nabla\phi_{p,\x}}^p}^{\frac{2}{p}}
				\leq C\int_0^1\int_\Omega\frac{1-t}{|\j u_{\x}+t\nabla\phi_{p,\x}|^{2-p}}|\nabla\phi_{p,\x}|^2 \dif t;
			\end{equation}
			we claim that the right-hand side goes to 0.
			By a Taylor expansion, we have, since \(p \le 2\) if either \(\j u_{\x} \ne 0\) or \(\nabla \phi_{p,\x} \ne 0\),
			\begin{equation}
			\label{eq_Ash3eeFoh0Chosit1aighaeX}
				\begin{split}
					|&\nabla\phi_{p,\x}+\j u_{\x}|^p-|\j u_{\x}|^p-p\frac{\j u_{\x}}{|\j u_{\x}|^{2-p}}\cdot{\nabla\phi_{p,\x}}\\
					&= p\int_0^1\brk{1 - t}\frac{\abs{\j u_{\x}+t\nabla \phi_{p,\x}}^2\abs{\nabla\phi_{p,\x}}^2-\brk{2-p}\brk{\brk{\j u_{\x}+t\nabla \phi_{p,\x}}\cdot\nabla\phi_{p,\x}}^2}{\abs{\j u_{\x}+t\nabla \phi_{p,\x}}^{4-p}} \dif t \\
					&\geq p(p-1)\int_0^1\int_\Omega\frac{\brk{1-t}\abs{\nabla\phi_{p,\x}}^2}{|\j u_{\x}+t\nabla \phi_{p,\x}|^{2-p}} \dif t.
				\end{split}
			\end{equation}
			The inequality still holds when \(\j u_{\x} = \nabla \phi_{p,\x} \ne 0\) by convexity of the map \(s \in \Rset^2 \mapsto \abs{s}^p\).
			
			Again, using $\phi=0$ as a competitor in \eqref{eq: minimisation} we get  
			\begin{equation*}
				\int_\Omega|\nabla\phi_{p,\x}+\j u_{\x}|^p\leq \int_\Omega|\j u_{\x}|^p
			\end{equation*}
			and consequently
			\begin{equation}\label{eq: intermediate step}
			(p-1)\int_0^1\int_\Omega\frac{\brk{1-t} \abs{\nabla\phi_{p,\x}}^2 }{|\j u_{\x}+t\nabla \phi_{p,\x}|^{2-p}}\dif t\ \leq \abs[\Big]{\int_\Omega\frac{\j u_{\x}}{|\j u_{\x}|^{2-p}}\cdot\nabla\phi_{p,\x}}.
			\end{equation}
			We study the right-hand side of \eqref{eq: intermediate step} by relying on two different arguments to handle the regime close to the singularities and the one far. 
			Let $0<\delta'<\delta$ and $0\leq \eta\leq 1$ be a smooth cut-off function, identically equal to $1$ in $B_{\delta'}(\x_*)$ and vanishing outside of $B_{2{\delta'}}(\x_*)$; we break down 
			\begin{equation*}
			\begin{split}
				\int_\Omega\frac{\j u_{\x}}{|\j u_{\x}|^{2-p}}\cdot\nabla\phi_{p,\x}&=\int_\Omega\frac{\j u_{\x}}{|\j u_{\x}|^{2-p}}\cdot\nabla((1-\eta)\phi_{p,\x}) + \int_\Omega\frac{\j u_{\x}}{|\j u_{\x}|^{2-p}}\cdot\nabla(\eta\phi_{p,\x})\\
				&=\mathrm{I}+\mathrm{II}.
			\end{split}
			\end{equation*}
			Note that integration in $\mathrm{I}$ happens in $\Omega_{\delta'}\coloneqq\Omega\setminus \bigcup_{j = 1}^n  B_{\delta'}(x_{j})$, where $\j u_{\x}$ is smooth and
			\begin{equation*}
				\|\j u_{\x}\|_{L^\infty(\Omega_{\delta'})}\leq C.
			\end{equation*}
			Using that $\j u_{\x}$ is divergence free, we get 
			\begin{equation*}
			\begin{split}
				&\abs[\Big]{\int_{\Omega_\delta}\frac{\j u_{\x}}{|\j u_{\x}|^{2-p}}\cdot\nabla((1-\eta)\phi_{p,\x})}\\
				&\qquad =\abs[\Big]{\int_{\Omega_\delta}\left(\frac{\j u_{\x}}{|\j u_{\x}|^{2-p}}-\j u_{\x}\right)\cdot\nabla((1-\eta)\phi_{p,\x})}\\
				&\qquad \leq C\brk{\norm{\eta}_{L^\infty} + \norm{\nabla \eta}_{L^\infty}}\brk[\Big]{\int_{\Omega_{\delta'}}
				\abs[\Big]{\frac{\j u_{\x}}{|\j u_{\x}|^{2-p}}-\j u_{\x}}^{\frac{p}{p - 1}}}^{1 - \frac{1}{p}} \left(\int_{\Omega}|\nabla\phi_{p,\x}|^p\right)^{\frac1p}
			\end{split}
			\end{equation*}
			where we used Poincaré inequality on $\phi_{p,\x}$. On the other hand, by the mean value theorem
			\begin{equation}
			\label{eq_ook1duaZie0theew5leiH4no}
			\begin{split}
				\abs[\Big]{\frac{\j u_{\x}}{|\j u_{\x}|^{2-p}}-\j u_{\x}}
				&= \abs[\big]{\abs{\j u_{\x}}^{p-1}-\abs{\j u_{\x}}}
				= \int_{p - 1}^1 \abs{\j u_{\x}}^{q} \abs{\ln \abs{\j u_{\x}}} \dif q \\
				&\le \brk{2 - p}\max \brk{\abs{\j u_{\x}}^{p - 1}, \abs{\j u_{\x}}}
				\abs{\ln \abs{\j u_{\x}}}\\
				& \le \sum_{j = 1}^n C \brk{2 - p}\frac{\brk{1 + 
				\abs{\ln \abs{x - x_j}}}}{\abs{x - x_j}},
			\end{split}
			\end{equation}
			since \(p - 1 < 1\).
			Noting that $\|\j u_{\x}\|_{L^\infty(\Omega_{\delta'})} < \infty$, we find
			\begin{equation*}
				\int_{\Omega_{\delta'}}
				\abs[\Big]{\frac{\j u_{\x}}{|\j u_{\x}|^{2-p}}-\j u_{\x}}^{\frac{p}{p - 1}}
				\leq C\brk{2-p}.
			\end{equation*}
			It follows that
			\begin{equation}
			\label{eq_Vaiph2ahngiejauquo5iex6c}
				\mathrm{I}= \abs[\Big]{\int_{\Omega_\delta}\frac{\j u_{\x}}{|\j u_{\x}|^{2-p}}\cdot\nabla((1-\eta)\phi_{p,\x})}
				\leq C\brk{2-p}
				\brk[\Big]{
				\int_{\Omega} \abs{\nabla \phi_{p, \x}}^p
				}^\frac{1}{p}
			\end{equation}
			uniformly with respect to $\x\in B_\delta(\x_*)$.
			
			To show that $\mathrm{II}\to 0$, since $\j u_{\x}$ is divergence-free
			we have 			
			\begin{equation}
			\label{eq_be3Iziequieheiheequahl4s}
				\operatorname{div}\brk[\Big]{\frac{\j u_{\x}}{|\j u_{\x}|^{2-p}}}=- \frac{\brk{2 - p} \Inpr{\nabla \j u_{\x}}{\j u_{\x}\otimes \j u_{\x}}}{\abs{\j u_{\x}}^{4 - p}}.
			\end{equation}
			By \eqref{eq: expansion vx} and \eqref{eq_be3Iziequieheiheequahl4s}, one estimates directly, if \(\j u_{\x} \ne 0\) in \(\bigcup_{j = 1}^n B_{2\delta'}\brk{x_j}\),
			\begin{equation*}
				\left|\operatorname{div}\left(\frac{\j u_{\x}}{|\j u_{\x}|^{2-p}}\right)\right|\leq \frac{C}{|x-x_j|^{p-1}}\quad\text{ in }B_{2\delta'}(x_j).
			\end{equation*}
			Integrating by parts, using H\"older and Poincaré inequality, we find
			\begin{equation}
			\label{eq_doh7Aika2Saig1gohv1ahcho}
			\begin{split}
				\abs[\Big]{\int_{B_{2{\delta'}}\brk{x_j}}\frac{\j u_{\x}}{|\j u_{\x}|^{2-p}}\cdot\nabla(\eta\phi_{p,\x})}
				&\leq (2-p)\sum_{j=1}^n\int_{B_{2\delta'}(x_j)}\frac{C}{|x-x_j|^{p-1}}\phi_{\x,p}\brk{x} \dif x\\
				&\leq C(2-p)\brk[\Big]{\int_{B_{2{\delta'}}}\frac{1}{|x|^{p}} \dif x}^{\frac{p-1}{p}}\brk[\Big]{\int_\Omega|\nabla \phi_{p,\x}|^p}^{\frac1p}\\
				&\leq C(2-p)^{\frac{1}{p}} \brk[\Big]{\int_\Omega|\nabla \phi_{p,\x}|^p}^{\frac1p}.
			\end{split}
			\end{equation}
Combining \eqref{eq_phaib9eewieThee5mirahqu1}, \eqref{eq: intermediate step}, \eqref{eq_Vaiph2ahngiejauquo5iex6c} and \eqref{eq_doh7Aika2Saig1gohv1ahcho}, we get
\begin{equation}
 			\brk[\Big]{\int_{\Omega}\abs{\nabla\phi_{p,\x}}^p}^{\frac{2}{p}}
 			\le C \brk{2 - p}^\frac{1}{p}  \brk[\Big]{\int_\Omega|\nabla \phi_{p,\x}|^p}^{\frac1p},
\end{equation}
from which the conclusion follows.
\end{proof}

		Now, let $\chi$ be a smooth cut-off function identically equal to 1 in $B_1$ and vanishing in $\Rset^2 \setminus B_2$. 
		For a fixed choice of indices $(j,\ell)\in\{1,\dots,n\}\times\{1,2\}$ and a generic point $\x=(x_1,\dots,x_n)\in\Omega^n_*$, we let 
		\begin{equation*}
			X_{j,\ell}
			\defeq 
			\chi\left(\frac{x-x_j}{\delta}\right)e_\ell
		\end{equation*}
		and remark that $X_{j,\ell}\equiv e_\ell$ in $B_\delta(x_j)$, while it vanishes outside of  $B_{2\delta}(x_j)$. Next, given an $\h=(h_{j,\ell})_{j,\ell\in\{1,\dots,n\}\times\{1,2\}}\in(\R^2)^n$ we set 
		\begin{equation}\label{eq: X_h}
			X_{\h}\coloneqq\sum_{j,\ell}X_{j,\ell}h_{j,\ell}
		\end{equation}
		
	\begin{proposition}\label{prop: DW}
		For every $\x\in\Omega^n_*$ and any increment $\h\in\R^{2n}$, we have
		\begin{equation}\label{eq: lemhopf}
			DW(\x)[\h]=\int_\Omega\langle S_2(u_{\x}),DX_{\h}\rangle.
		\end{equation}
	\end{proposition}
	
	\begin{remark}
		Note that while $S_2(u_{\x})\notin L^1(\Omega)$, it belongs to $L^1(\spt DX_{\h})$ and the left-hand side in \eqref{eq: lemhopf} makes sense. 
	\end{remark}
	
	The proof of Proposition \ref{prop: DW} requires two preliminary results. First, recall that we defined the renormalised energy in \eqref{eq_ahSh8shaiphaiha3oocheich} as 
\begin{equation*}
	W(\x)\coloneqq \lim_{\rho\to 0}W_\rho(\x)
\end{equation*}
where
\[
W_{\rho}\brk{\x}\defeq 
\int_{\Omega \setminus \bigcup_{j = 1}^n B_\rho \brk{x_n}} \abs{\nabla u_{\x}}^2 - 2 \pi \sum_{j = 1}^n d_j^2 \ln \frac{1}{\rho}.
\]
We will prove that the derivatives $DW_\rho(\x)$ exist and converge uniformly in $\Omega$, thus proving that $W$ is differentiable in $\Omega$ and that 
\begin{equation*}
	DW(\x)\coloneqq \lim_{\rho\to 0}DW_\rho(\x)
\end{equation*} 
Let $\h$ be an element of $\R^{2n}$ and let $\Phi_t$ be the flow associated to \eqref{eq: X_h}, namely the solution to  
%$X\in C^\infty_c(\Omega,\R^2)$ and $\varphi^X_t$ is the associated flow, solving 
\begin{equation*}
\left\{
\begin{aligned}
	\partial_t\Phi^{X_{\h}}_t&=X_{\h}\circ\Phi^{X_{\h}}_t,\\
	\Phi_0^{X_{\h}}&=\operatorname{id}.
\end{aligned}
\right.
\end{equation*}
Let
\begin{equation*}
	\x_t\defeq \x+t\h
\end{equation*}
and $u_{\x_t}$ the corresponding harmonic map. Using the flow we can define 
\begin{equation*}
	\tilde u_{\x}\coloneqq u_{\x_t}\circ \Phi_{t}^{-1}.
\end{equation*}
and note that 
\begin{equation}\label{eq: lifting psit}
	\tilde u_{\x}=e^{i\psi_{t,\x}} u_{\x}
\end{equation}
for some $\psi_{t,\x} \in W^{1,2}\brk{\Omega, \Rset}$, since they are homotopically equivalent. The first lemma is the usual inner variation equation, slightly generalised to functions having a $t$-dependence themselves. 
	\begin{lemma}\label{lem: inner variation with t}
	Let $\Phi_t$ be the flow associated to a vector field $X\in C^\infty_c(\Omega)$. Let $u_t$ be a family of $W^{1,2}(\Omega)$ functions depending from a parameter $t$ in a $C^1$ way and that are constant in $t$ outside of $\spt DX$. Then,
		\begin{equation*}
		\frac{d}{dt}\Big\vert_{t=0}\int_{\Omega}|\nabla (u_t\circ\Phi_t)|^2=\int_{\Omega}\langle S_2(u_0),DX\rangle+2\int_{\Omega}\nabla u_0\cdot\nabla (\partial_t\vert_{t=0}u_t).
	\end{equation*}
	\begin{proof}
		A change of variable yields
		\begin{equation*}
			\int_{\Omega}|\nabla (u_t\circ\Phi_t)|^2=\int_{\Omega}\sum_{i}\Big(\sum_j\partial_ju_t(y)\partial_i\Phi^j_t(\Phi_t^{-1}(y))\Big)\det D\Phi_t^{-1}(y) \dif y
		\end{equation*}
		and the result follows by direct differentiation. 
	\end{proof}
	\end{lemma}	
	\begin{lemma}\label{lem: psit estimate}
	\label{lemma_composition_harmonic}
	Let $\psi_{t}$ be the phase defined by \eqref{eq: lifting psit}. Then 
	\begin{equation*}
		\sup_{\x\in B_\delta(\x_*)} \abs{\partial_t\vert_{t=0}\psi_{t,\x}}\leq C|\h|
	\end{equation*}
	for some constant $C>0$.
	\begin{proof}
		Out of notational simplicity we denote $\psi_{t,\x}$ simply by $\psi_t$. The claim follows from standard elliptic estimates once we show that $\partial_t\vert_{t=0}\psi_t$ satisfies an elliptic equation. Using the $\s^1$-harmonicity of $u_{\x}$, we have 
		\begin{equation*}
			\Delta\psi_t=\div \j\tilde u_{\x}=\div \j (u_{\x_t}\circ \Phi_{-t}) 
		\end{equation*}
		Direct differentiation yields, using Einstein summation convention and denoting $v=u_{\x_t}$ and $\Phi=\Phi_{-t}$ for simplicity,
		\begin{align*}
		\div \j(v\circ\Phi)&=\partial_j(\partial_k(v\circ\Phi)\partial_j\Phi^k)\\
			&=\partial_{k\ell}(v\circ\Phi)\partial_{k}\Phi^j\partial_{\ell}\Phi^j+\partial_k(v\circ\Phi)\partial_{jj}\Phi^k.
		\end{align*}
		In other words
		\begin{equation*}
			\div \j(u_{\x_t}\circ \Phi_{-t})=\operatorname{tr}\left((D\Phi_{-t})^T(D^2u_{\x_t}(\Phi_{-t}),iu_{\x_t}(\Phi_{-t})) (D\Phi_{-t})\right)+\j u_{\x_t}(\Phi_{-t})\cdot\Delta\Phi_{-t}.
		\end{equation*}
		Next, we differentiate the expression above in $t$ and evaluate in $t=0$, finding 
		\begin{equation*}
			\begin{cases}
				\Delta(\partial_t\psi_t\vert_{t=0})=\eta_{\x}&\text{in }\Omega\\
				\partial_t\psi_t\vert_{t=0}=0&\text{on }\partial \Omega
			\end{cases}
		\end{equation*} 
		where 
		\begin{equation*}
			\eta_{\x}\coloneqq \tfrac{d}{dt}\big\vert_{t=0}\div \j(u_{\x_t}\circ \Phi_{-t}).
		\end{equation*}
		We claim that $|\eta_{\x}|\leq C|\h|$ for a constant $C$ independent on $\x\in B_\delta(\x_*)$, thus concluding the proof. Using similar calculations to the ones above and recalling that $\div \j u_{\x_t}=0$ in $\Omega$ for every $t$, as well as the fact that $\Phi_{-t}=\operatorname{id}-tX+O(t^2)$, we find
		\begin{equation*}
			\eta_{\x}=-2\langle(D^2u_{\x},iu_{\x}), DX\rangle -(Du_{\x},iu_{\x})\cdot \Delta X.
		\end{equation*}
		Note that $\eta_{\x}$ is supported away from $\x$. By \eqref{eq: X_h}, we find 
		\begin{equation*}
			\sup_{\x\in B_\delta(\x_*)}|\eta_{\x}|\leq C|\h|
		\end{equation*} 
		for some constant $C=C(\delta)>0$.
	\end{proof}
	\end{lemma}
	With Lemma \ref{lem: inner variation with t} and Lemma \ref{lem: psit estimate}, we can finally prove Proposition \ref{prop: DW}.
	\begin{proof}[Proof of Proposition \ref{prop: DW}]
		As mentioned above, our goal is to prove uniform convergence of the derivatives $DW_{\rho}(\x)$. Fixing a vector $\h\in\R^{2n}$ as above and using the same notations we compute, by a change of variable and using Lemma \ref{lem: inner variation with t}  
\begin{align*}
	\frac{d}{dt}\Big\vert_{t=0}W_\rho(\x_t)&=\frac{d}{dt}\Big\vert_{t=0}\int_{\Omega\setminus B_\rho(\x_t)}|\nabla u_{\x_t}|^2\\
	&=\frac{d}{dt}\Big\vert_{t=0}\int_{\Omega\setminus B_\rho(\x_t)}|\nabla (\tilde u_{\x}\circ\Phi_t)|^2\\
	&=\int_{\Omega\setminus B_\rho(\x)}\langle S_2(u_{\x}),DX\rangle+2\int_{\Omega\setminus B_\rho(\x)}\nabla u_{\x}\cdot\nabla (\partial_t\vert_{t=0}\tilde u_{\x}) \\
	&=\int_{\Omega}\langle S_2(u_{\x}),DX\rangle+2\int_{\Omega\setminus B_\rho(\x)}\j u_{\x}\cdot\nabla (\partial_t\vert_{t=0}\tilde \psi_t)
\end{align*}
where in the last step we also used the fact that, for $\rho<\delta$, $X_{\h}$ is constant on $B_\rho(\x)$. Next, we claim that 
\begin{equation*}
	\lim_{\rho\to 0}\sup_{\x\in B_\delta(\x_*)}\abs[\Big]{\int_{\Omega\setminus B_\rho(\x)}\j u_{\x}\cdot\nabla (\partial_t\vert_{t=0}\tilde \psi_t)} =0
\end{equation*}
thus concluding the proof. Integrating by parts and using Lemma \ref{lem: psit estimate}, we find 
	\begin{equation*}
		\abs[\Big]{\int_{\Omega\setminus B_\rho(\x)}\j u_{\x}\cdot\nabla (\partial_t\vert_{t=0}\tilde \psi_t)}
		=\abs[\Big]{\int_{\partial B_\rho(\x)}(\partial_t\psi_t)\j u_{\x}\cdot \nu}
		\leq C\abs{\h}\abs[\Big]{\int_{\partial B_{\rho}}\j u_{\x}\cdot \nu}	\end{equation*}
	the conclusion follows from the fact that, as $\rho$ becomes smaller, $\j u_{\x}$ becomes arbitrarily close to being tangential to $\partial B_\rho(\x)$.
	\end{proof}

Thanks to Proposition \ref{prop: DW}, we can prove Proposition~\ref{proposition_critical}.
\begin{proof}[Proof of Proposition~\ref{proposition_critical}]
It follows from the assumptions that \(S_{p_k} \brk{u_k} \to S_2 \brk{u}\) in \(L^1_{\mathrm{loc}}\brk{\Omega\setminus \set{x_{1, *}, \dotsc, x_{n, *}}}\) and thus Proposition~\ref{prop: DW} implies that \(W \brk{x_*} = 0\).
\end{proof}
	
	With Lemma \ref{lem: aux} and Proposition \ref{prop: DW} we have all the instruments to prove Proposition \ref{prop: unif conv}.
	
	\begin{proof}[Proof of Proposition \ref{prop: unif conv}] First, note that 
		\begin{equation}\label{eq: divS}
			\div S_p(u_{p,\x})(X_{j,\ell})=c_j^\ell(p,\x).
		\end{equation}
		Our goal is to show that 
		\begin{equation}\label{eq: claim}
			\int_\Omega \left\langle S_p(u_{p,\x}), DX_{j,\ell}\right\rangle\to \int_\Omega \left\langle S_2(u_{\x}), DX_{j,\ell}\right\rangle
		\end{equation}
		as $p\to 2$, uniformly in $\x\in B_\delta(\x_*)$, for any $(j,\ell)\in\{1,\dots,n\}\times\{1,2\}$. By \eqref{eq: divS} and Proposition \ref{prop: DW} this would conclude the proof. Define the operator $R_{p,\x}$ by
		\begin{align*}
			S_p(u_{p,\x})&=p|\j u_{p,\x}|^{p-2}\j u_{p,\x}\otimes \j u_{p,\x}-|\j u_{p,\x}|^pI\\
			&=p|\nabla\phi_{p,\x}+\j u_{\x}|^{p-2}(\nabla\phi_{p,\x}+\j u_{\x})\otimes (\nabla\phi_{p,\x}+\j u_{\x})-|\nabla\phi_{p,\x}+\j u_{\x}|^pI\\
			&\eqqcolon R_{p,\x}(\nabla\phi_{p,\x}),
		\end{align*}
		for $p\leq 2$, and observe that \eqref{eq: claim} holds if
		\begin{equation*}
			\int_\Omega\langle R_{p,\x}(\nabla\phi_{p,\x})-R_{2,\x}(0),DX_{j,\ell}\rangle\to 0
		\end{equation*}
		uniformly in $\x\in B_\delta(\x_*)$. We prove this by showing that 
		\begin{equation}\label{eq: claim_1}
			\int_\Omega\langle R_{p,\x}(\nabla\phi_{p,\x})-R_{p,\x}(0),DX_{j,\ell}\rangle\to 0
		\end{equation}
		and 
		\begin{equation}\label{eq: claim_2}
			\int_\Omega\langle R_{p,\x}(0)-R_{2,\x}(0),DX_{j,\ell}\rangle\to 0.
		\end{equation}
		The validity of \eqref{eq: claim_2} follows directly once we notice that the integration is happening away from singularities. To show \eqref{eq: claim_1}, we can estimate 
		\begin{equation*}
			|R_{p,\x}(\nabla\phi_{p,\x})-R_{p,\x}(0)|\leq C|\nabla\phi_{p,\x}+\j u_{\x}|^{p-1}|\nabla\phi_{p,\x}|
		\end{equation*}
		with a constant $C$ independent on $p$ and $\x\in\mathcal{U}$. Using the boundedness -- uniform over $\x \in B_\delta(\x_*)$ -- of $|\j u_{\x}|$ on $\spt DX_{j,\ell} $ and H\"older inequality, we get 
		\begin{equation*}
			\left|\int_\Omega\langle R_{p,\x}(\nabla\phi_{p,\x})-R_{p,\x}(0),DX_{j,\ell}\rangle\right|\leq C\int_{\spt DX_{j,\ell}}|\nabla\phi_{p,\x}|^p
		\end{equation*}
		which vanishes as $p\to 2$ by Lemma \ref{lem: aux}. The proof is concluded. 
	\end{proof}

	\appendix

	\section{Renormalised energy identities}

\subsection{Differentiating the renormalised energy}	

The derivative of the renormalised energy can be expressed simply in terms of \(u_{\x}\) (see \cite[Theorem VIII.3]{Bethuel-Brezis-Helein1994}).
	\begin{proposition}
	\label{proposition_derivative_phase}
	One has
	\begin{equation}
	\label{eq_ekah1their3aw9wee1UaGhah}
	  \nabla W (\x)
	  = 2 \pi \brk{d_1 \j u_1 \brk{x_1}^\perp, \dotsc, d_n \j u_n \brk{x_n}^\perp}.
	\end{equation}
	where
	\begin{equation}
	\label{eq_Poo3reibuthiu0iecah1yes1}
	 u_j \brk{x} = \frac{\abs{x - x_j}^{d_j}}{\brk{x - x_j}^{d_j}} u_{\x} \brk{x}.
	\end{equation}
	\end{proposition}
	\begin{proof}
		We have, taking the limit where \(\chi\) tends to \(\chi_{B_{\rho} \brk{x_j}}\) in Proposition~\ref{prop: DW}
			\begin{equation}
			\begin{split}
	&\nabla W (\x) \cdot (h_1, \dotsc, h_n)\\
	 &\qquad  =
	  \sum_{i = 1}^n
	  \lim_{\rho \to 0} \int_{\partial B_{\rho} \brk{x_j}}
	  \brk{\j u_{\x} \cdot \nu}{\j u_{\x} \cdot h_j}
	  - \frac{\abs{\j u_{\x}}^2}{2} \brk{h_j \cdot \nu}.
	  \end{split}
	\end{equation}
By \eqref{eq_Poo3reibuthiu0iecah1yes1}, we have
	\begin{equation}
	\label{eq_sephee9uug1Chahgeijebea9}
	\j u_{\x} \brk{x} =
	 \j u_j\brk{x}
	 + d_j \frac{\brk{x - x_j}^\perp}{\abs{x - x_j}^2},
	\end{equation}
and thus for every \(x \in \partial B_\rho \brk{x_j}\),
	\[
	 \j u_{\x} \brk{x} \cdot h_j  
	 = \j u_j\brk{x} \cdot h_j +
	 d_j \frac{\det \brk{x - x_j, h_j}}{\abs{x - x_j}^2}
	\]
and thus
	\[
	 \begin{split}
	   &\lim_{\rho \to 0} \int_{\partial B_{\rho} \brk{x_j}} \brk{\j u_{\x} \cdot \nu}{\j u_{\x} \cdot h_j}
	  - \frac{\abs{\j u_{\x}}^2}{2} \brk{h_j \cdot \nu}\\
	  &\qquad =  d_j
	  \lim_{\rho \to 0} \int_{\partial B_{\rho} \brk{x_j}}
	  \brk{\j u_j \cdot \nu}{\tau \cdot h_j} - \brk{\j u_j^{-1} \nabla u_j \cdot \tau}{\nu \cdot h_j}\\
	  &\qquad =  d_j
	  \lim_{\rho \to 0} \int_{\partial B_{\rho} \brk{x_j}}
	  \brk{\j u_j^\perp \cdot \tau}{\tau \cdot h_j} + \brk{\j u_j^\perp \cdot \nu}{\nu \cdot h_j}\\
	  &\qquad = 2 \pi d_j \j u_j\brk{x_j}^\perp \cdot h_j,
	 \end{split}
	\]
	which is \eqref{eq_ekah1their3aw9wee1UaGhah}.
	\end{proof}

\subsection{Renormalised energy and Green functions}
\label{appendix_renormalised}
We relate the definition of the renormalized energy in terms of Dirichlet energy of maps outside small disks to the Dirichlet energy of suitable solutions of a linear elliptic problem with singular data.

The singular harmonic map problem with those given singularities consists in finding \(v \in C^2 \brk{\Omega\setminus \set{x_1, \dotsc, x_n}, \Sset^1}\) such that 
	\begin{equation}
	\label{eq_Jie1airaewieG2cophiew3iu}
	\left\{
	\begin{aligned}
	 \operatorname{div} \brk{\j v} &= 0 &&\text{in \(\Omega \setminus \set{x_1, \dotsc, x_n}\)},\\
	 v& = g &&\text{on \(\partial \Omega\)},\\
	 \deg_{x_j} v &= d_j && \text{for \(i \in \set{1, \dotsc ,n}\)},
	\end{aligned}
	\right.
	\end{equation}
where \(\deg_{x_j} v\) denotes the degree of the map \(v\vert_{\partial B_r\brk{x_j}}\) for \(r\) small enough so that \(\Bar{B}_r \brk{x_j}\subseteq \Omega \setminus \set{x_1, \dotsc, x_{i - 1}, x_j, \dotsc, x_n}\), that is 
\[
 \deg_{x_j} v = \frac{1}{2 \pi i}\int_{\partial B_r \brk{x_j}} v^{-1} \partial_{\tau} v,
\]
where \(\tau\) is the tangential vector field in the anti-clockwise direction.

We consider the singular linear elliptic problem 
of finding \(\Phi \in C^2 \brk{\Omega\setminus \set{x_1, \dotsc, x_n},\R}\) and $\Theta\in C^2(\Omega,\R)$ such that 
	\begin{equation}
	\label{eq_aireecajivei9ge2dePhahgh}
	 	\left\{
	\begin{aligned}
	 \Delta \Phi &= \sum_{i = 1}^n 2 \pi d_j \delta_{x_j} & &\text{in \(\Omega\)},\\
	 \Delta \Theta & = 0&& \text{in \(\Omega\)},\\
	 \Theta\vert_{\Gamma_\ell} &\text{ is constant}&& \text{for \(i \in \set{1, \dotsc , m}\)}\\ 
	 \partial_{\nu} \Phi & = - ig^{-1} \partial_{\tau} g && \text{on \(\partial \Omega\)},\\
	 g \brk{\gamma \brk{1}} &=  e^{i\mathfrak{F}_{\gamma} \brk{\Phi, \Theta} } g \brk{\gamma \brk{0}}
	 & &\text{for every \(\gamma \in \Sigma\)}\
	\end{aligned}
	\right.
	\end{equation}
	where 
	\[
	\mathfrak{F}_{\gamma} \brk{\Phi, \Theta}
	\defeq 
	\int_0^1 \det\brk{\nabla \Phi \brk{\gamma \brk{t}}, \gamma' \brk{t}} \dif t
	 + \Theta \brk{\gamma \brk{1}} - \Theta \brk{\gamma \brk{0}}
	\]
	and 
	\[
	\Sigma \defeq \set{\gamma \in C^1 \brk{[0, 1], \Omega \setminus \brk{x_1, \dotsc, x_n}} \st \gamma \brk{0}, \gamma \brk{1} \in \partial \Omega}.
	\]

	Here we use the convention that \(\nu\) is the \emph{outgoing normal} to \(\partial \Omega\) and
	\(\tau\) the tangential component taken so that \(\det \brk{\nu, \tau} = 1\), that is \(\tau\) gives an anti-clockwise rotation on the outer component of the boundary and clockwise on the inner components.

	\begin{proposition}
	\label{proposition_equivalence_singular_harmonic}
	One has 
	\begin{multline*}
	 \set{-\j v \st v \in C^2 \brk{\Omega \setminus \set{x_1, \dotsc, x_n}, \Sset^1} \text{ solves } \eqref{eq_Jie1airaewieG2cophiew3iu} }\\
	= \set{\nabla^\perp \Phi  + \nabla \Theta \st \Phi \in C^2 \brk{\Omega \setminus \set{x_1, \dotsc, x_n}, \Rset} \text{ and } \Theta \in C^2 \brk{\Omega, \Rset}\text{ solve \eqref{eq_aireecajivei9ge2dePhahgh}} }.
	\end{multline*}
	Moreover, solutions are characterized by the quantities
	\begin{align*}
	  \int_{\Gamma_\ell} i v^{-1}\partial_\nu v
	  & &
	  \text{and }& &
	  	  \int_{\Gamma_\ell} \partial_\nu \Theta
	  &&\text{for \(\ell \in \set{1, \dotsc, m}\)},
	\end{align*}
	in the sense that the bijection in the first part of the statement can be taken to preserve their value.
	\end{proposition}
	
	Here we write
	\[
	 \nabla^\perp \Phi = \brk{-\partial_2 \Phi, \partial_1 \Phi},
	\]
	so that we have  
	\[
	  \nabla^\perp \Phi \cdot h = \partial_1 \Phi h_2 - \partial_2 \Phi h_1
	  = \det \brk{\nabla \Phi, h}.
	\]
	and on the boundary
	\[
	  \nu \cdot \nabla^\perp \Phi = - \tau \cdot \nabla \Phi.
	\]
	
	\begin{proof}[Proof of Proposition~\ref{proposition_equivalence_singular_harmonic}]
	Assume that \(\brk{\Phi, \Theta}\) solves the linear problem \eqref{eq_Jie1airaewieG2cophiew3iu}.
	We fix some point \(x_0 \in \partial \Omega\) and we define 
	\[
	 \Sigma_* \defeq \set{\gamma \in C^1 \brk{[0, 1], \Omega \setminus \{x_1, \dotsc, x_n\}} \st \gamma \brk{0} = x_0}
	\]
	By the divergence theorem, if \(\gamma\in \Sigma_*\), then 
	\(
	 \mathfrak{F}_{\gamma} \brk{\Phi, \Theta} \in 2\pi \Zset,
	\)
	since it follows from \eqref{eq_aireecajivei9ge2dePhahgh} that the flux of \(\nabla \Phi\) around any \(x_j\) or \(\Gamma_\ell\) is an integer multiple of \(2 \pi\).
	It follows thus that one can set 
	\[
	 v \brk{x} = e^{i\mathfrak{F}_{\gamma} \brk{\Phi, \Theta} }  g \brk{x_0} 
	\]
	for any \(\Sigma_*\) such that \(\gamma\brk{1} = x\).
	We have then for every \(h \in \Rset^2\),
	\[
	 \nabla v \brk{x} \cdot h = i v \brk{x}\brk{\det \brk{\nabla \Phi \brk{x}, h} + \nabla \Theta \brk{x} \cdot h},
	\]
	and thus 
	\[
	 \j v = \nabla^\perp \Phi  + \nabla \Theta.
	\]
	Checking that,
	\[
	 \deg_{x_j} v = \frac{1}{2\pi i} \int_{\partial B_r \brk{x_j}} v^{-1} \partial_{\tau} v = \frac{1}{2\pi} \int_{\partial B_r \brk{x_j}} \partial_\nu \Phi = d_j,
	\]
	if \(\nu\) is the outward normal to \(B_r \brk{x_j}\).
	We have thus proved that the mapping \(v\) solves \eqref{eq_Jie1airaewieG2cophiew3iu}.

	Conversely, if \(v\) satisfies \eqref{eq_Jie1airaewieG2cophiew3iu},
	we choose \(\Theta\) to be a solution of the problem
	\begin{equation}
	\left\{
	\begin{aligned}
	 \Delta \Theta &= 0 &&\text{in \(\Omega\)},\\
	 \Theta & = 0 &&\text{on \(\partial \Gamma_0\)},\\
	\Theta\vert_{\Gamma_\ell} &\text{ is constant}&& \text{for \(i \in \set{1, \dotsc , m}\)}\\ 
	\int_{\partial \Gamma_\ell} \partial_\nu \Theta
	&= - \int_{\partial \Gamma_\ell} i v^{-1}\partial_\nu v&& \text{for \(i \in \set{1, \dotsc , m}\)},
	\end{aligned}
	\right.
	\end{equation}
	and we take then \(\Phi\) so that 
	\[
	 \nabla^\perp \Phi = \j v  - \nabla \Theta,
	\]
	which is well defined up to a constant. The function \(\Phi\) has then the required properties.
	\end{proof}

	\begin{proposition}
	\label{proposition_renormalised_Green_Formula}
	If \(\Phi\) and \(\Theta\) satisfy \eqref{eq_aireecajivei9ge2dePhahgh}, then 
	\begin{equation}
	\label{eq_ooRooquel5kai7oowe3jephe}
	\begin{split}
	 &\lim_{\rho \to 0}
	 \int_{\Omega \setminus \bigcup_{j = 1}^n B_\rho \brk{x_i}}
	 \abs{\nabla^\perp \Phi + \nabla \Theta}^2 - \sum_{j = 1}^n 2 \pi d_j^2 \ln \frac{1}{\rho} \\
	 &\qquad =  
	 - \int_{\partial \Omega} \Phi \brk{i g^{-1} g'}
	 - \sum_{j = 1}^n 2 \pi d_j H_j \brk{x_j}
	 + \int_{\partial \Omega} \Theta \, \partial_\nu \Theta\\
	 &\qquad = \smashoperator[r]{\sum_{\substack{1 \le i, j \le n\\ i \ne j}}} 2 \pi d_i d_j \ln \frac{1}{\abs{x_i - x_j}}
	 + \int_{\partial \Omega}  \Phi \brk{i g^{-1} g'}
	 - \sum_{j = 1}^n 2 \pi d_j H_* \brk{x_j} 
	  + \int_{\partial \Omega} \Theta \, \partial_\nu \Theta.
	 \end{split}
	\end{equation}
	where \(H_j \defeq C \brk{\Omega \setminus \set{x_1, \dotsc, x_{j - 1}, x_{j + 1}, \dotsc, x_n}, \Rset}\) is defined so that for every \(x \in \Omega \setminus \set{x_1, \dotsc, x_n}\)
\begin{equation}
\label{eq_Meil7shai1rai0Bie5vueCh5}
H_j \brk{x} \defeq \Phi \brk{x} - d_j \ln \abs{x - x_j}.
\end{equation}
and  \(H_* \defeq C \brk{\Omega \setminus \set{x_1, \dotsc, x_n}, \Rset}\) is defined so that for every \(x \in \Omega \setminus \set{x_1, \dotsc, x_n}\)
\[
 H_* \brk{x} \defeq \Phi \brk{x} - \sum_{j = 1}^n d_j \ln \abs{x - x_j}.
\]
	\end{proposition}

When \(\Omega\) is simply connected, or equivalently, \(m = 0\), one has \(\Theta= 0\) and one recovers the formula of Bethuel, Brezis and Hélein \cite[Theorem I.7]{Bethuel-Brezis-Helein1994};
our expression is similar to the formula of Rodiac and Ubillús \cite[Theorem 1.1]{Rodiac_Ubillus_2022} who perform an additional minimization over the singular harmonic maps having some prescribed boundary datum and singularities.
	
\begin{proof}[Proof of Proposition~\ref{proposition_renormalised_Green_Formula}]
Since \(\Theta\) is harmonic, we have by integration by parts:
	\begin{equation}
	\label{eq_Shae7aijai8Wee1autuz5ze2}
	 \lim_{\rho \to 0}
	 \int_{\Omega \setminus \bigcup_{j = 1}^n B_\rho \brk{x_i}}
	 \abs{\nabla \Theta}^2
	 = \int_{\Omega}\abs{\nabla \Theta}^2
	 = \int_{\partial \Omega} \Theta \, \partial_\nu \Theta.
	 \end{equation}

Next, since \(\nabla \Phi \in W^{1, 1}_{\mathrm{loc}}\), we have, since \(\nabla \cdot \brk{\nabla^\perp \Theta} = 0\) in \(\Omega\) and \(\nabla^\perp \Theta \cdot \nu =  -\tau \cdot \nabla \Theta\)
\begin{equation}
\label{eq_Giuch8Booreiyiech3lapuoG}
\begin{split}
 \lim_{\rho \to 0}
 	 \int_{\Omega \setminus \bigcup_{j = 1}^n B_\rho \brk{x_j}}
	 \nabla \Theta \cdot
	 \nabla^\perp \Phi
	 & =
	 \int_{\Omega} \nabla \Theta \cdot
	 \nabla^\perp \Phi = -  \int_{\Omega} \nabla^\perp \Theta  \cdot
	 \nabla \Phi\\
	 & =
	 \int_{\partial \Omega}
	 \Phi \brk{\nabla^\perp \Theta \cdot \nu}
	 - \int_{\Omega} \Phi
	  \nabla \cdot \brk{\nabla \Theta}^\perp =  0.
\end{split}
\end{equation}
Finally, we have
\begin{equation}
\label{eq_doo0UnohwugheMoo6xai0pai}
  \int_{\Omega \setminus \bigcup_{j = 1}^n B_\rho \brk{x_i}}
	 \abs{\nabla \Phi}^2
	 =
	 \int_{\partial \Omega} \Phi \nabla \Phi \cdot \nu
	 - \sum_{i = 1}^n \int_{\partial B_{\rho} \brk{x_i}} \frac{\brk{x - x_i} \cdot \nabla \Phi \brk{x}}{\rho} \Phi \brk{x} \dif x.
\end{equation}
We compute by \eqref{eq_aireecajivei9ge2dePhahgh}
\[
  \int_{\partial \Omega} \Phi \nabla \Phi \cdot \nu
  = - \int_{\partial \Omega} \Phi \brk{i g^{-1} g'}.
\]
On the other hand we have
\[
 \int_{\partial B_{\rho} \brk{x_j}} \frac{\brk{x - x_j} \cdot \nabla \Phi \brk{x}}{\rho} \dif x
 = 2\pi d_j
\]
and by \eqref{eq_Meil7shai1rai0Bie5vueCh5}
\[
 \int_{\partial B_{\rho} \brk{x_j}} \frac{\brk{x - x_j} \cdot \nabla \Phi \brk{x}}{\rho} \Phi \brk{x} \dif x
 + 2\pi d_j^2 \ln \frac{1}{\rho} 
 =  \int_{\partial B_{\rho} \brk{x_j}} \frac{\brk{x - x_j} \cdot \nabla \Phi \brk{x}}{\rho} H_j \brk{x} \dif x,
\]
One has then, by the divergence theorem
\begin{equation}
 \label{eq_Woo0wae1daaz7Aephaelohph}
 \int_{\partial B_{\rho} \brk{x_j}} \frac{\brk{x - x_j} \cdot \nabla \Phi \brk{x}}{\rho} H_j \brk{x} \dif x
 = \int_{B_\rho \brk{x_j}} \nabla \Phi \cdot \nabla H
  + 2 \pi d_j H_j \brk{x_j}
\end{equation}
We get thus  by \eqref{eq_doo0UnohwugheMoo6xai0pai} and by \eqref{eq_Woo0wae1daaz7Aephaelohph}
\begin{equation}
\label{eq_ieCu2aeNung7kie6aiZeixei}
\lim_{\rho \to 0}
	 \int_{\Omega \setminus \bigcup_{j = 1}^n B_\rho \brk{x_i}}
	 \abs{\nabla^\perp \Phi}^2 - 
	  \sum_{j = 1}^n 2\pi d_j^2 \ln \frac{1}{\rho} \\
	 =  - \int_{\partial \Omega} \Phi \brk{i g^{-1} g'}
	 - \sum_{j = 1}^d 2 \pi d_j H_j \brk{x_j}
\end{equation}
Combining \eqref{eq_Shae7aijai8Wee1autuz5ze2}, \eqref{eq_Giuch8Booreiyiech3lapuoG} and \eqref{eq_ieCu2aeNung7kie6aiZeixei}, we get \eqref{eq_ooRooquel5kai7oowe3jephe}.
\end{proof}

We finally get as a consequence of Proposition \ref{proposition_derivative_phase},
a formula for the derivative of the renormalised energy, extending the formula for the simply-connected case of Bethuel, Brezis and Helein \cite[Theorem VIII.3]{Bethuel-Brezis-Helein1994}.

	\begin{proposition}
	\label{proposition_derivative_Green}
	One has
	\begin{equation}
	  \nabla W (\x)
	  = - 4 \pi \brk{d_1 \brk{\nabla H_1 \brk{x_1}- \nabla^\perp \Theta \brk{x_1}}, \dotsc,d_n \brk{\nabla H_n \brk{x_n} - \nabla^\perp \Theta \brk{x_n}}}.
	\end{equation}
	\end{proposition}
	\begin{proof}
	If \(v_j\) is given for \(j \in \set{1, \dotsc, n}\) by \eqref{eq_Poo3reibuthiu0iecah1yes1}, then we have by \eqref{eq_sephee9uug1Chahgeijebea9}
	\[
	\begin{split}
	  \j  u_j\brk{x}^{\perp}
	  &= -\nabla \Phi\brk{x} + \nabla^\perp \Theta\brk{x}
	  + d_j \frac{x - x_j}{\abs{x - x_j}^2}\\
	  &= -\nabla H_j \brk{x} + \nabla^\perp \Theta \brk{x}.
	  \end{split}
	\]
	We conclude by Proposition~\ref{proposition_derivative_phase}.
	\end{proof}

\bibliography{Bibliography} 
\bibliographystyle{amsalpha}

\end{document}